\UseRawInputEncoding
\documentclass[11pt,twoside,reqno]{amsart}
\usepackage{url}
\usepackage{mathrsfs}
\usepackage{latexsym}
\usepackage{amsmath}
\usepackage{amssymb}
\usepackage{amsthm}
\usepackage{amscd}
\usepackage{amsfonts}
\usepackage{epsfig}
\usepackage{multicol}
\usepackage{CJK}
\usepackage{bbm}
\usepackage{amsmath, amsthm, amsfonts, amssymb,mathtools}
\usepackage{mathrsfs}
\usepackage{enumitem} 

\allowdisplaybreaks
   
\numberwithin{equation}{section}

\theoremstyle{plain}
\newtheorem{thm}{Theorem}[section]
\newtheorem{lem}[thm]{Lemma}
\newtheorem{prop}[thm]{Proposition}
\newtheorem{rem}[thm]{Remark}
\newtheorem{cor}[thm]{Corollary}

\newtheorem{conj}[thm]{Conjecture}
\newcommand\restr[2]{{
		\left.\kern-\nulldelimiterspace 
		#1 
		\right|_{#2} 
}}

\newcommand{\pf}{\noindent\begin {proof}}
\newcommand{\epf}{\end{proof}}
\makeatletter

\begin{document}
	\title[Integrality of trigonometric determinants]{Integrality of a trigonometric determinant arising from a conjecture of Sun}

\author{Liwen Gao}
\address{Liwen Gao: Department of Mathematics, Nanjing University,
	Nanjing 210093, China;  gaoliwen1206@smail.nju.edu.cn}

\author{Xuejun Guo$^{\ast}$}
\address{Xuejun Guo$^{\ast}$: Department of Mathematics, Nanjing University,
	Nanjing 210093, China;  guoxj@nju.edu.cn}
\address{$^{\ast}$Corresponding author}

\thanks{The authors are supported by National Natural Science Foundation of China (Nos. 11971226, 12231009).}

\date{}
\maketitle

\noindent

\begin{abstract} 
In this paper we resolve a conjecture of Zhi-Wei Sun concerning the integrality and arithmetic structure of certain trigonometric determinants. 
Our approach builds on techniques developed in our previous work, where trigonometric determinants were studied via special values of Dirichlet $L$-functions. The method is refined by establishing a connection between odd characters modulo $4n$ and even characters modulo $n$. The results highlight a close connection between trigonometric determinant matrices, Fourier-analytic structures, and arithmetic invariants.
\end{abstract}
\medskip

\textbf{Keywords:} Determinants,  $L$-functions, Dirichlet characters.

\vskip 10pt

\textbf{2020 Mathematics Subject Classification:} 11C20,  11M20, 47A10.

	\section{Introduction}
	Determinants of matrices with trigonometric entries have attracted considerable attention in recent years, due to their rich algebraic structure and unexpected connections with number theory. Such determinants often arise in problems related to character sums, discrete Fourier analysis, and special values of arithmetic functions.
	
	In a series of conjectures, Zhi-Wei Sun \cite{s1} proposed striking integrality and arithmetic properties of determinants formed from trigonometric functions evaluated at rational multiples of $\pi$ in 2019.
	Two conjectures among them were as follows, 
	\begin{conj}\label{Conjsec} Let $p$ be an odd prime. Define
		$$D(p):=\det\left[\sec \left(\frac{2\pi jk}{p}\right)\right]_{0\le j,k\le(p-1)/2}.$$
		If $p\equiv 1\pmod4$, then $D(p)=0$. When $p\equiv3\pmod4$, the number
		$$\frac{D(p)}{2^{(p-3)/2}(-p)^{(p+1)/4}}$$ is a positive odd integer
		\end{conj}
		
		\begin{conj}\label{Conjcsc}
		We have
		$$c_p:=\frac{1}{2^{(p-1)/2}p^{(p-5)/4}}\det\left[\csc\left(\frac{2\pi jk}{p}\right)\right]_{1\le j,k\le(p-1)/2}\in\mathbb{Z}.$$
		Moreover, $c_p=1$ if $p\equiv3\pmod8$, and $c_p=0$ if $p\equiv7\pmod 8$.
		\end{conj}

More generally, Sun has also investigated integrality and arithmetic
phenomena of various determinants and related combinatorial quantities
from a broader perspective; see for example \cite{s3,s4}.
These works provide important background and motivation for the present paper.
		
The Conjecture \ref{Conjcsc} were resolved by the authors in \cite{gg} through a new method. In this paper, we prove the Conjecture \ref{Conjsec} by our approach.

However, our previous method only adapts to the situation of the odd characters and a direct application of the methods in~\cite{gg}
does not suffice. So we need to modify our method.

The key refinement introduced in this paper is the establishment of a precise connection between odd Dirichlet characters modulo $4n$ and even Dirichlet characters modulo $n$. This correspondence allows us to deal with the situation of even characters .

Using this refined framework, we are able to resolve Conjecture~\ref{Conjsec} by proving a stronger structural statement. As a consequence, the predicted integrality follows as a special case. The precise formulation is given in Theorem~\ref{secthm}
and its corollaries.

Let $S_n=\{a\thinspace |\thinspace 1\le a<\frac{n}{2},(a,n)=1\}$ and $U_n=\{a\thinspace |\thinspace -\frac{n}{2}< a<\frac{n}{2},(a,n)=1\}$.
Let $$ D_n=\bigg(\sec(\tfrac{2jk\pi}{n})\bigg)_{j,k\in S_n}.$$  Our main result is as follows.
	\begin{thm}\label{secthm}
	Let $n\geq 3$ be a positive odd integer and $m=\frac{\phi(n)}{2}$. Then
	$$
	\det(D_n)=\varepsilon(n)(-1)^{\frac{(n-1)m}{2}}\bigg(\frac{2n}{\pi}\bigg)^m\prod_{\chi_n\, \text{even}}\chi_n(4)L(1,\chi_n\psi)$$
	where $\chi_n$ run thorugh all Dirichlet even characters modulo $n$, $\psi$ is the unique primitive Dirichlet character  modulo $4$ satisfying $\psi(n)=(\frac{-1}{n})$ and $\varepsilon(n)$ is as Lemma \ref{sign}.
\end{thm}

	When $n=p$ be odd prime, we have the following corollary.
\begin{cor}\label{seccor1} Let $p$ be an odd prime. Let $\ell$ be the order of $4$ in the multiplicative group $(\mathbb Z/p\mathbb Z)^\times$. Then
	$$\det(D_{p})=	
	\begin{cases}
		\displaystyle \bigg(\frac{2}{p}\bigg)2^{(p-5)/2}(p+1)p^{(p-3)/4}\frac{h_{4p}^-}{h_{p}^-}
		&\text{if} \enspace p\equiv 3 \pmod{4}\,\\
		\displaystyle -\bigg(\frac{2}{p}\bigg)2^{(p-5)/2}(p-1)p^{(p-3)/4}\frac{h_{4p}^-}{h_{p}^-}, & \text{if}\enspace p\equiv 1 \pmod{4}\, \text{and}\enspace 4\nmid \ell,\\[6pt]
		\displaystyle	(-1)^{\frac{p-1}{l}+1}\bigg(\frac{2}{p}\bigg)2^{(p-5)/2}(p-1)p^{(p-3)/4}\frac{h_{4p}^-}{h_{p}^-}, & \text{if} \enspace p\equiv 1 \pmod{4}\, \text{and}\enspace 4|\,\ell ,\\
	\end{cases}$$
where $h_{4p}^-$ is divisible by $h_{p}^-$.
\end{cor}	
\begin{rem}
This result resolves an open question raised by Zhi-Wei Sun in an online mathematical forum.
\end{rem}
		Let $$D^{(1)}_p:=\left(\sec\left(\frac{2\pi jk}{p}\right)\right)_{0\le j,k\le(p-1)/2}.$$
	The following corollary, which follows from Theorem \ref{secthm}, implies Conjecture \ref{Conjsec}.
	\begin{cor}\label{seccor2} Let $p$ be an odd prime. Then
		$$\det(D^{(1)}_p)=	
		\begin{cases}
			\displaystyle \bigg(\frac{2}{p}\bigg)2^{(p-3)/2}p^{(p+1)/4}\frac{h_{4p}^-}{h_{p}^-}
			&\text{if} \enspace p\equiv 3 \pmod{4}\,\\
			\displaystyle 0, & \text{if}\enspace p\equiv 1 \pmod{4}.\\[6pt]
		\end{cases}$$
		where $h_{4p}^-$ is divisible by $h_{p}^-$.
	\end{cor}	
	
	\begin{rem}
		Our results show that Sun’s Conjecture~\ref{Conjsec} in its original form is not valid in full generality.
		Indeed, for $p=31$ one obtains $D_p=5084$ via {\tt Sagemath}, which contradicts the conjecture that $D_p$ is a positive odd integer.
		A refined and correct formulation is given in Corollary \ref{seccor2}.
		\end{rem}
		Our results highlight a close connection between trigonometric determinant matrices, discrete Fourier analysis, and arithmetic invariants arising from Dirichlet characters and $L$-functions. We expect that the techniques developed here may be applicable to other classes of structured matrices with arithmetic symmetry.
		
	\section{preliminaries}
	
Before we start to prove the main theorem, we establish some preliminary results. 

\begin{lem}\label{sign}
	Let $n= \prod\limits_{i=1}^r p_i^{e_i}$ be an odd  positive integer where $p_1,\dots,p_r$ are distinct odd primes and $g$ be an even fucntion defined on $\mathbb{Z}$ satisfying with period $n$ (i.e. $g(i+n)=g(i)$ for all $i\in\mathbb{Z}$). Define
	\begin{equation*}
		F_n=\big(g(jk)\big)_{j,k\in S_n},\qquad F_n'=\big(g(jk')\big)_{j,k\in S_n}.
	\end{equation*}
	Then  
	$$\det(F_n)=\varepsilon(n)\det(F'_n),$$
	where $$\varepsilon(n)=
	\begin{cases}
		-1&r=1\ \text{and}\ p_1\equiv1\ \text{or}\ 4e_1+3\pmod8,\\
		-1&r=2\ \text{and}\ p_1+p_2\equiv0\pmod4,\\
		1& \text{otherwise}.
	\end{cases}$$
	In particular, when $n=p$ is an odd prime, $\varepsilon(p)=-\left(\tfrac{2}{p}\right).$
\end{lem}
\begin{proof}
	Let $k\in S_n$ and let  $\tilde k\in S_n$ be the unique integer such that $k\tilde k\equiv \pm1 \pmod{n}$ and define $\tau_n$: $S_n\rightarrow S_n$ by $\tau_n(k)=\tilde k$. Obviously,  $\tau_n$ is a permutation of $S_n$. If $n = \prod\limits_{i=1}^r p_i^{e_i}$ is an odd integer where $p_1,\dots,p_r$ are distinct odd primes, then by Theorem 1.1 of \cite{s2}, 
$$\varepsilon(n):=\operatorname*{sign}(\tau_n(k))=
\begin{cases}
	-1&r=1\ \text{and}\ p_1\equiv1\ \text{or}\ 4e_1+3\pmod8,\\
	-1&r=2\ \text{and}\ p_1+p_2\equiv0\pmod4,\\
	1& \text{otherwise}.
\end{cases}$$ 

Since $g$ is an even function, we obtain $(F'_n)$ from $(F_n)$ by reordering the columns according to the permutation $(k\mapsto \tilde k)$. Then
	$$\det(F'_n)=\varepsilon(n)\det(F_n).$$
	In other words, $\det(F_n)=\varepsilon(n)\det(F'_n)$. In particular, when $n=p$ is an odd prime, a direct computation shows that
	$$\varepsilon(p)=\operatorname*{sign}(\tau_p(k))=-\left(\tfrac{2}{p}\right).$$
	
\end{proof}

When $\chi$ is not primitive, we rely on the following tools to address this case.

	\begin{lem}[\cite{pp}]\label{char}
		Let $\chi $ be any Dirichlet character $\bmod{q}$. Then there exist a unique
		divisor $q^\ast\mid q$ and a unique \emph{primitive} character
		$\chi^\ast $ $\bmod{q^{\ast}}$such that for any $n$ coprime to $q$,
		\begin{equation}\label{eq:agree-on-coprime}\tag{14}
			\chi(n)=\chi^\ast(n).
		\end{equation}
		Conversely, for any divisor $q^\ast\mid q$ and any primitive character
		$\chi^\ast \bmod{q^\ast}$, there exists a unique character $\chi \bmod q$
		for which \eqref{eq:agree-on-coprime} holds whenever $(n,q)=1$. In fact,
		\begin{equation}\label{eq:product-with-principal}\tag{15}
			\chi(n)=\chi^{\ast}(n)\,\chi_{0}(n),
		\end{equation}
		where $\chi_{0}\bmod{q}$ denotes the principal character .
	\end{lem}

		\begin{lem}[\cite{yw}]\label{Lfunction}If $\chi$ is a non-principal character$\pmod n$ with conductor $f_{\chi}$ induced by the primitive character $\chi$ in the sense of Lemma \ref{char}, then
		$$L(s,\chi)=L(s,\chi^{*})\prod\limits_{\substack{p\mid n\\p\,\nmid \, f_{\chi}}}\biggl(1-\frac{\chi^{*}(p)}{p^s}\biggr)$$
	\end{lem}

	\begin{lem}[\cite{kt}]\label{cot}
		For a Dirichlet odd character $\chi\bmod n$, we have
		$$\sum_{j\in U_n}\chi(j)\cot(\frac{j\pi}{n})=\frac{2n}{\pi}L(1,\chi).$$
	\end{lem}

\begin{lem}\label{tanL}
	Let $n$ be a positive odd integer. $\chi_n$ is an even character $\pmod{n}$ and $\psi$ is the character modulo $4$. Then $\chi'=\chi\psi$ is a character modulo $4n$ and we have
	$$\sum_{j\in U_{4n}}\chi'(j)\tan(\frac{j\pi}{4n})=\frac{8n}{\pi}L(1,\chi').$$
\end{lem}
\begin{proof}	It follows that from Lemma \ref{cot},
	\begin{align*}
	\sum_{j\in U_{4n}}\chi(j)\tan(\frac{j\pi}{4n})&=-\sum_{j\in U_{4n}}\chi'(j)\cot(\frac{(j+2n)\pi}{4n})\\
	&=-\sum_{j\in U_{4n}}\chi'(j'-2n)\cot(\frac{j'\pi}{4n})\\
	&=\sum_{j\in U_{4n}}\chi'(j)\cot(\frac{j\pi}{4n})\\
	&=\frac{8n}{\pi}L(1,\chi')
	\end{align*}
	\end{proof}	
	
We now establish a formula that expresses a secant sum in terms of the values of Dirichlet $L$-functions at $s=1$.

\begin{lem}\label{secL}
Let $\psi$ be the character satisfying $\psi(n)=(\frac{-1}{n})$, $\chi_n$ be a Dirichlet even character $\chi_n\bmod n$ and we denote $\chi_n\psi$ by $\chi'$. Then we have
$$\sum_{j\in U_n}\chi_n(j)\sec(\frac{2j\pi}{n})=
	\frac{(-1)^{\frac{n-1}{2}}4n\chi_n(4)}{\pi}L(1,\chi')
$$
	\end{lem}
	
\begin{proof}
Since
		$$
	\sec(\frac{2j\pi}{n})
	= \frac{1}{2}\bigl(\tan(\frac{(4j+n)\pi}{4n})
	- \tan(\frac{(4j-n)\pi}{4n})\bigr),
	$$	
	it follows from Lemma \ref{tanL} that,  when $n \equiv 1 \pmod{4}$,
		\begin{align*}
		\sum_{j\in U_n}\chi(j)\sec(\frac{2j\pi}{n})&=
		 \frac{1}{2}\sum_{j\in U_n}\chi(j)\bigl(\tan(\frac{(4j+n)\pi}{4n})
		- \tan(\frac{(4j-n)\pi}{4n})\bigr)\\
&=\frac{1}{2}\sum_{j\in U_n}\chi(j)\bigl(\tan(\frac{(4j+n)\pi}{4n})
- \tan(\frac{(4j-n)\pi}{4n})\bigr)\\		
&=\frac{1}{2}\chi(4)\sum_{{\substack{j\in U_{4n}\\ j\equiv n\pmod{4}}}}\chi(j)\tan(\frac{j\pi}{4n})\\
&\quad -\frac{1}{2}\chi(4)\sum_{{\substack{j\in U_{4n}\\ j\equiv -n\pmod{4}}}}\chi(j)\tan(\frac{j\pi}{4n})\\
&=\frac{1}{2}\chi(4)\sum_{{\substack{j\in U_{4n}}}}(-1)^{\frac{j-1}{2}}\chi(j)\tan(\frac{j\pi}{4n})\\
&=\frac{1}{2}\chi(4)\sum_{{\substack{j\in U_{4n}}}}\psi'(j)\tan(\frac{j\pi}{4n})\\
&=\frac{4n\chi(4)}{\pi}L(1,\chi').\\
				\end{align*}

	Similarly, when $n\equiv 3\pmod{4}$,
$$
	\sum_{j\in U_n}\chi(j)\sec(\frac{2j\pi}{n})		
	=-\frac{4n\chi(4)}{\pi}L(1,\chi').
$$
So  $$\sum_{j\in U_n}\chi_n(j)\sec(\frac{2j\pi}{n})=
 \frac{(-1)^{\frac{n-1}{2}}4n\chi_n(4)}{\pi}L(1,\chi').$$
\end{proof}

		\begin{prop}[\cite{w2}]\label{h}
	Let $\{\chi_1,\ldots,\chi_m\}$ be the set of all odd Dirichlet characters
	modulo $n$ and $K=\mathbb{Q}(\zeta_n)$. Let $K^+$ be its maximal real subfield. Then 
	$$\prod_{\chi \thinspace\text{odd}}L(1,\chi^*)=\frac{(2\pi)^{\frac{\phi(n)}{2}}h_n^-}{Qw\sqrt{\mid d(K)/d(K^+)\mid}}=\frac{(2\pi)^{\frac{\phi(n)}{2}}h_n^-}{Qw\sqrt{\prod\limits_{\chi \text{odd}}f_{\chi}}},
	$$
	where  $\chi^*$ stands for the primitive character associated with $\chi$ in the sense of Lemma \ref{char}, $w$ is the number of roots of unity in $\mathbb{Q}(\zeta_n)$, $f_{\chi}$ is the conductor of the character $\chi$,  $h_n^-$ denotes the relative class number of $\mathbb{Q}(\zeta_n)$ and
	$$Q=\begin{cases}
		1&\text{when n is a prime power}  ,\\
		2&\text{otherwise}.\\
	\end{cases}$$
\end{prop}

		\begin{prop}\label{4h}
	Let $p$ be an odd prime and $\{\chi_1,\ldots,\chi_m\}$ be the set of all even Dirichlet characters modulo $p$. Then 
	$$\prod_{\chi \thinspace\text{even}}L(1,\psi\chi)=\begin{cases}
		\frac{\pi^{(p-1)/2}(p-1)h_{4p}^-}{4p^{(p+1)/4}h_{p}^-}&\,\text{when}\,\, p\equiv 1\pmod{4}\\
		\frac{\pi^{(p-1)/2}(p+1)h_{4p}^-}{4p^{(p+1)/4}h_{p}^-}	&\,\text{when}\,\, p\equiv 3\pmod{4}\\
	\end{cases}.
	$$

\end{prop}
\begin{proof}
	We denote by $\chi_0 \bmod p$ the principal Dirichlet character.
	
	\medskip
	Applying Proposition~\ref{h} with $n=4p$, we obtain
	\[
	\prod_{\substack{\chi \bmod 4p\\ \chi\ \mathrm{odd}}}
	L(1,\chi^\ast)
	=\frac{2^{(p-7)/2}\pi^{p-1}h_{4p}^-}{p^{p/2}}.
	\]
	Similarly, applying Proposition~\ref{h} with $n=p$, we have
	\[
	\prod_{\substack{\chi \bmod p\\ \chi\ \mathrm{odd}}}
	L(1,\chi^\ast)
	=\frac{2^{(p-3)/2}\pi^{(p-1)/2}h_p^-}{p^{(p+3)/4}}.
	\]
	
	\medskip
	We now describe the structure of odd Dirichlet characters modulo $4p$.
	Every such character can be uniquely written as one of the following three types:
	\begin{enumerate}
		\item[(i)] $\chi_p'\chi_0$, where $\chi_p'$ is an odd character modulo $p$ and
		$\chi_0$ is the principal character modulo $4$;
		\item[(ii)] $\chi_p\psi$, where $\chi_p$ is an even character modulo $p$ and
		$\psi$ is the unique odd character modulo $4$;
		\item[(iii)] $\chi_0\psi$, where $\chi_0$ is the principal character modulo $p$
		and $\psi$ is as in {\rm (ii)}.
	\end{enumerate}
	
	\medskip
	By Lemma~\ref{char}, the primitive character associated to a character of type
	{\rm (i)} is $\chi_p'$. The characters of type {\rm (ii)} are primitive.
	The character of type {\rm (iii)} is induced from $\psi$.
	Moreover, $L(1,\psi\chi_{0})=L(1,\psi)(1-\frac{\psi(p)}{p})$,
where $$(1-\frac{\psi(p)}{p})=
\begin{cases}
	\frac{p-1}{p}&\,\text{when}\,\, p\equiv 1\pmod{4},\\
	\frac{p+1}{p}&\,\text{when}\,\, p\equiv 3\pmod{4}.\\
\end{cases}$$

Combining the above decompositions with the two product formulas, we deduce that
\begin{align*}
\prod_{\chi \thinspace\text{even}}L(1,\psi\chi)&=\frac{\prod\limits_{\substack{\chi\bmod{4n}\\ \chi\thinspace\text{odd}}}L(1,\chi^*)}{\prod\limits_{\substack{\chi\bmod{n} \\\chi\thinspace\text{odd}}}L(1,\chi^*)}(1-\frac{\psi(p)}{p})\\
&=\begin{cases}
	\frac{\pi^{(p-1)/2}(p-1)h_{4p}^-}{4p^{(p+1)/4}h_{p}^-}&\,\text{when}\, \,p\equiv 1\pmod{4},\\
	\frac{\pi^{(p-1)/2}(p+1)h_{4p}^-}{4p^{(p+1)/4}h_{p}^-}	&\,\text{when}\, \, p\equiv 3\pmod{4}.\\
\end{cases}\\
\end{align*}
	\end{proof}
	
\begin{lem}\label{chi4}
	Let $\ell$ be order of $4$ in $(\mathbb Z/p\mathbb Z)^\times$. Then
	\[
	\prod_{\chi_p \text{even}}\chi(4)
	=
	\begin{cases}
		\displaystyle 1, & \text{if $4\nmid\ell$ },\\[6pt]
		\displaystyle (-1)^{\frac{\ell}{p-1}}, & \text{if $4\mid\ell$ }.
	\end{cases}
	\]
\end{lem}

\begin{proof}
	If $\ell$ is odd, then $\ell\,\mid\,(p-1)/2$ and
	\[
	\prod_{\chi\, \text{even}}\bigl(\chi(4)\bigr)
	=\prod_{i=1}^{l-1}\zeta_\ell^{\frac{i(p-1)}{2\ell}}
	=(-1)^{\frac{(p-1)(l-1)}{2\ell}}=1,
	\]
	where $\zeta_\ell$ is a primitive $\ell$-th root of unity.

	If $\ell$ is even, then $\ell\,\mid\,p-1$. Hence
	\[
.	\prod_{\chi\, \text{even}}\bigl(\chi(4)\bigr)
=\prod_{i=1}^{(l/2-1)}\zeta_\ell^{\frac{2i(p-1)}{\ell}}
=(-1)^{\frac{(\ell-2)(p-1)}{2\ell}}=	\begin{cases}
	\displaystyle 1, & \text{if $\ell\equiv 2\pmod{4}$ },\\[6pt]
	\displaystyle (-1)^{\frac{\ell}{p-1}}, & \text{if $\ell\equiv 0\pmod{4}$. }
\end{cases}
	\]

\end{proof}
\begin{lem}[\cite{m}]\label{relativeclass}
	Let $h_n^-$ be relative class number of $\mathbb{Q}(\zeta_n)$. If $m\mid n$,
	then $h_m^-\mid h_n^-$.
	\end{lem}
	\section{Proof of the Main Results}
We now turn to the proof of our main results.
		
\begin{proof}[Proof of Theorem \ref{secthm}]
	Let $$D'_n=\bigg(\sec(\frac{2jk'\pi}{n})\bigg)_{j,k\in S_n}.$$ By Lemma \ref{sign},
$\det(D_n)=\varepsilon(n)\det(D'_n)$. so it suffices to compute $\det(D'_n)$.
	 Let $\{\chi_0,\ldots,\chi_m\}$ be the set of all even Dirichlet characters
	modulo $n$, where $m=\phi(n)/2$. Let $S_n=\{a_1,\ldots,a_m\}$ and set
	\[
	\Omega_n=m^{-1/2}\,[\chi_i(a_j)]_{1\le i,j\le m}.
	\]
	By the orthogonality relations for Dirichlet characters, it follows that $\Omega$ is a unitary matrix, and
	\[
	\sum_{{\substack{\chi\bmod{n}\\ \chi\,\text{even}}}}\chi(a)=
	\begin{cases}
		\frac{\phi(n)}{2},& a=1,\,n-1\\
		0,& \text{otherwise},
	\end{cases}
	\qquad
	\sum_{a\in S_n}\chi_i(a)\,\overline{\chi_j(a)}=\frac{\phi(n)}{2}\,\delta_{ij}.
	\]
	Note that
	\[
	\det\!\Bigl(\sec\!\bigl(\tfrac{2ab'\pi}{n}\bigr)\Bigr)_{a,b\in S_n}
	=\det\!\Bigl(\Omega_n\;(\sec\!\bigl(\tfrac{2ab'\pi}{n}\bigr))_{a,b\in S_n}\;\Omega_n^\ast\Bigr)
	=\det\!\Bigl(\tfrac{1}{m}\,(s_{ij})_{1\le i,j\le m}\Bigr),
	\]
	where
	\[
	s_{ij}:=\sum_{a\in S_n}\sum_{b\in S_n}\chi_i(a)\,
	\sec\!\left(\tfrac{2ab'\pi}{n}\right)\overline{\chi_j(b)}.
	\]
One can show that by Lemma \ref{secL},
	\begin{align*}
		&\sum_{a\in(\mathbb Z/n\mathbb Z)^\times}\ \sum_{b\in(\mathbb Z/n\mathbb Z)^\times}
		\chi_i(a)\,\sec(\frac{2ab'\pi}{n})\,\overline{\chi_j(b)}\\
		&=\sum_{c\in(\mathbb Z/n\mathbb Z)^\times}\ \sum_{b\in(\mathbb Z/n\mathbb Z)^\times}
		\chi_i(cb)\,\sec(\frac{2c\pi}{n})\,\overline{\chi_j(b)}\\
		&=\left(\sum_{b\in(\mathbb Z/n\mathbb Z)^\times}\chi_i(b)\,\overline{\chi_j(b)}\right)
		\left(\sum_{c\in(\mathbb Z/n\mathbb Z)^\times}\sec(\frac{2c\pi}{n})\,\chi_i(c)\right)\\
		&=\,(-1)^{\frac{n-1}{2}}\frac{4n\phi(n)\delta_{ij}\chi_n(4)}{\pi}L(1,\chi').
	\end{align*}
	
Next, we consider
	\begin{align*}
		\sum_{a\notin S_n}\ \sum_{b\in S_n}\chi_i(a)\,\sec(\frac{2ab'\pi}{n})\,\overline{\chi_j(b)}
		&=\sum_{a\in S_n}\ \sum_{b\in S_n}\chi_i(-a)\,\sec(\frac{-2ab'\pi}{n})\,\overline{\chi_j(b)}\\
		&=s_{ij}.
	\end{align*}
By the same argument as above, we obtain
	\begin{align*}
		\sum_{a\in S_n}\ \sum_{b\notin S_n}\chi_i(a)\,\sec(\frac{2ab'\pi}{n})\,\overline{\chi_j(b)}
		&=s_{ij},\\
		\sum_{a\notin S_n}\ \sum_{b\notin S_n}\chi_i(a)\,\sec(\frac{2ab'\pi}{n})\,\overline{\chi_j(b)}&=s_{ij}.
	\end{align*}
	Hence 
	$$s_{ij}=\,(-1)^{\frac{n-1}{2}}\frac{n\phi(n)\delta_{ij}\chi_n(4)}{\pi}L(1,\chi_n\psi).$$
	Then
	$$\det(D'_n)=(-1)^{\frac{(n-1)m}{2}}\bigg(\frac{2n}{\pi}\bigg)^m\prod_{\chi_n\, \text{even}}\chi_n(4)L(1,\chi_n\psi).
	$$

So
\begin{align*}
\det(D_n)=\varepsilon(n)(-1)^{\frac{(n-1)m}{2}}\bigg(\frac{2n}{\pi}\bigg)^m\prod_{\chi_n\, \text{even}}\chi_p(4)L(1,\chi_n\psi).
\end{align*}

	\end{proof}
Meanwhile, we also give the proof of Corollary \ref{seccor1} and \ref{seccor2}.

\begin{proof}[Proof of Corollary \ref{seccor1}]
	When $p\equiv 3 \pmod{4}$, $4\nmid\ell$ since $\ell\mid p-1$. So by Lemma \ref{chi4},
	$$\prod_{\chi_n\, \text{even}}\chi_n(4)=1.$$
	Then by Theorem \ref{secthm}, Lemma \ref{4h},
	\begin{align*}
	\det(D_p)&=\bigg(\frac{2}{p}\bigg)\bigg(\frac{2p}{\pi}\bigg)^{(p-1)/2}\prod_{\chi_p\, \text{even}}\chi_p(4)L(1,\chi_p\psi),\\
	&=\bigg(\frac{2}{p}\bigg)2^{(p-5)/2}(p+1)p^{(p-3)/4}\frac{h_{4p}^-}{h_{p}^-}.\\
	\end{align*}
	
	By Theorem \ref{secthm}, Lemmas \ref{4h} and \ref{chi4}, when $p \equiv 1 \pmod{4}$, we have 
		\begin{align*}
		\det(D_p)&=-\bigg(\frac{2}{p}\bigg)\bigg(\frac{2p}{\pi}\bigg)^{(p-1)/2}\prod_{\chi_p\, \text{even}}\chi_p(4)L(1,\chi_p\psi),\\
		&=-\bigg(\frac{2}{p}\bigg)2^{(p-5)/2}(p-1)p^{(p-3)/4}\frac{h_{4p}^-}{h_{p}^-}\prod_{\chi_p\, \text{even}}\chi_p(4),\\
		&=	\begin{cases}
			\displaystyle -\bigg(\frac{2}{p}\bigg)2^{(p-5)/2}(p-1)p^{(p-3)/4}\frac{h_{4p}^-}{h_{p}^-}, & \text{if $4\nmid \ell$ },\\[6pt]
			\displaystyle	(-1)^{\frac{p-1}{l}+1}\bigg(\frac{2}{p}\bigg)2^{(p-5)/2}(p-1)p^{(p-3)/4}\frac{h_{4p}^-}{h_{p}^-}, & \text{if $4|\,\ell$ }.
		\end{cases}
	\end{align*}
	By Lemma \ref{relativeclass},  we have $h_p^- \mid h_{4p}^-$.
	\end{proof}

\begin{proof}[Proof of Corollary \ref{seccor2}]
		When $\chi_p=\chi_{0}$ in Lemma \ref{secL},  
	\begin{equation}\label{rsum}
	\sum_{j\in S_p}\sec(\frac{2j\pi}{p})=\frac{1}{2}\sum_{j\in U_p}\sec(\frac{2j\pi}{p})=
	\begin{cases}
		\frac{p-1}{2},&\, p\equiv 1\pmod{4},\\
		-\frac{p+1}{2},&\, p\equiv 3\pmod{4}.\\
	\end{cases}\end{equation}
	
	We distinguish two cases according to the congruence class of $p \pmod{4}$.

\medskip
\noindent
\textbf{Case 1: $p \equiv 1 \pmod{4}$.}
Denote by $R_i$ the $i$-th row of $D_p^{(1)}$. Clearly,
\[
R_1 = (1,\ldots,1).
\]
Adding the sum of the rows $R_3,\ldots,R_n$ to the second row does not change the determinant. 
Thus, we perform the row operation
\[
R_2 \longleftarrow R_2 + \sum_{i=3}^{n} R_i.
\]
By \eqref{rsum}, the second row becomes
\[
R_2 = \left( \tfrac{p-1}{2},\ldots,\tfrac{p-1}{2} \right).
\]
Consequently,
\[
R_2 - \tfrac{p-1}{2} R_1 = (0,\ldots,0),
\]
which implies that the rows of $D_p^{(1)}$ are linearly dependent. Therefore,
\[
\det(D_p^{(1)}) = 0.
\]

\medskip
\noindent
\textbf{Case 2: $p \equiv 3 \pmod{4}$.}
Let $\mathbf{1} = (1,\ldots,1)$. Then $D_p^{(1)}$ admits the block decomposition
\[
D_p^{(1)} =
\begin{bmatrix}
	1 & \mathbf{1} \\
	\mathbf{1} & D_p
\end{bmatrix}.
\]
Subtracting the first row from each of the rows $R_2,\ldots,R_n$ does not change the determinant. Hence, we apply the row operations
\[
R_i \longleftarrow R_i - R_1, \qquad i=2,\ldots,n.
\]
After these operations, the matrix takes the form
\[
\begin{bmatrix}
	1 & \mathbf{1} \\
	0 & D_p^{(2)}
\end{bmatrix},
\]
and therefore
\[
\det(D_p^{(1)}) = \det(D_p^{(2)}).
\]

A direct computation shows that
\[
D_p^{(2)} = \left( \sec \frac{2\pi jk}{p} - 1 \right)_{1 \le j,k \le (p-1)/2}.
\]

Next, define
\[
D_p^{(2)'} := \left( \sec \frac{2\pi j k'}{p} - 1 \right)_{1 \le j,k \le (p-1)/2}.
\]
We first compute the determinant of $D_p^{(2)'}$. 
As in the proof of Theorem~\ref{secthm}, both
\[
\Omega_p D_p \Omega_p^{\ast}
\quad \text{and} \quad
\Omega_p D_p^{(2)'} \Omega_p^{\ast}
\]
are diagonal matrices, and they coincide in all diagonal entries except possibly the $(1,1)$-entry.
Indeed, when at least one of $\chi_i$ or $\chi_j$ is non-principal, we have
\[
\sum_{a \in S_n} \sum_{b \in S_n} \chi_i(a)\,\overline{\chi_j(b)} = 0.
\]

Let $d_{11}$ and $d^{(2)}_{11}$ denote the $(1,1)$-entries of $\Omega_p D_p \Omega_p^{\ast}$ and $\Omega_p D_p^{(2)'} \Omega_p^{\ast}$, respectively. 
A direct calculation yields
\[
d_{11} = -\frac{(p-1)(p+1)}{4},
\qquad
d^{(2)}_{11} = -\frac{(p-1)p}{2}.
\]
Consequently,
\[
\frac{\det(D_p^{(2)})}{\det(D_p)}
=
\frac{\det(D_p^{(2)'})}{\det(D_p')}
=
\frac{d^{(2)}_{11}}{d_{11}}
=
\frac{2p}{p+1}.
\]

Hence by Corollary \ref{seccor2} when $p\equiv 3 \pmod{4}$,
$$\det(D^{(2)}_p)=\bigg(\frac{2}{p}\bigg)2^{(p-3)/2}p^{(p+1)/4}\frac{h_{4p}^-}{h_{p}^-}.$$
As in the proof of Corollary \ref{seccor1}, $h_p^- \mid h_{4p}^-$.
\end{proof}

\subsubsection*{\noindent\textbf{Funding}} {\small The authors are supported by National Nature Science Foundation of China (Nos. 11971226, 12231009).}
\subsubsection*{\noindent\textbf{Data Availablity}} {\small No addtional data is available.}

\subsection*{\normalfont\Large\bfseries Declarations}
\par
\subsubsection*{\noindent\textbf{\small Conflict of interest:}} {\small The author has no relevant financial or non-financial interests to disclose.}

\end{document}